\documentclass[a4paper,12pt]{amsart}

\usepackage[utf8]{inputenc}
\usepackage{enumerate}
\usepackage{latexsym}
\usepackage{amsfonts,amsmath}
\usepackage{amssymb}
\usepackage{amsthm,amscd}
\usepackage{epsfig,graphicx}
\usepackage{setspace} 
\usepackage{hyperref}
\usepackage{cite}
\usepackage{enumerate}
\usepackage{hyperref}
\usepackage{endnotes}
\usepackage{stmaryrd}
\usepackage{stackrel}
\usepackage{color}
\usepackage{setspace}

\theoremstyle{plain}

\newtheorem*{Key Lemma}{Key Lemma}
\newtheorem{Theorem}{Theorem}[section]
\newtheorem{Proposition}[Theorem]{Proposition}
\newtheorem{Corollary}[Theorem]{Corollary}
\newtheorem{Lemma}[Theorem]{Lemma}

\theoremstyle{definition}

\newtheorem{example}[Theorem]{Example}
\newtheorem{Remark}{Remark}

    \def\a{{\alpha}}
\def\b{{\beta}}

\def\Z{{\mathbb Z}}
\def\Q{{\mathbb Q}}
\def\R{{\mathbb R}}
\def\P{{\mathbb P}}

\def\deg{\text{deg}}
\def\ord{\text{ord}}
\def\div{\rm div}
\def\C{$\mathcal{C}$}
\def\c{\mathcal{C}}
\def\ID2{$(\text{ID}_2)$}
\def\IDn{$(\text{ID}_n)$}
\def\GE2{$(\text{GE}_2)$}
\def\GEn{$(\text{GE}_n)$}
\def\Id2{(\text{ID}_2)}

\def\Ge2{(\text{GE}_2)}

    \def\a{{\alpha}}
\def\b{{\beta}}

\def\deg{{\rm deg}}

\begin{document}

\title[]{Factorizations into idempotent factors \\ of matrices over Pr\"ufer domains}

\author{Laura Cossu, Paolo Zanardo}

\address{Laura Cossu, Dipartimento di Matematica ``Tullio Levi-Civita'', Via Trieste 63 - 35121 Padova, Italy}

\email{lcossu@math,unipd.it}

\address{Paolo Zanardo, Dipartimento di Matematica ``Tullio Levi-Civita'', Via Trieste 63 - 35121 Padova, Italy}

\email{pzanardo@math.unipd.it}



\thanks{Research supported by Dipartimento di Matematica ``Tullio Levi-Civita", Università di Padova under Progetto SID 2016 BIRD163492/16 “Categorical homological methods in the study of algebraic structures” and Progetto DOR1690814 “Anelli e categorie di moduli”. The first author is a member of the Gruppo Nazionale per le Strutture Algebriche, Geometriche e le loro Applicazioni (GNSAGA) of the Istituto Nazionale di Alta Matematica (INdAM)}

\begin{abstract}

A classical problem, that goes back to the 1960's, is to characterize the integral domains $R$ satisfying the property (ID$_n$): ``every singular $n\times n$ matrix over $R$ is a product of idempotent matrices''. Significant results in \cite{Laff1}, \cite{Ruit} and  \cite{BhasRao} motivated a natural conjecture, proposed by Salce and Zanardo \cite{SalZan}: (C) ``an integral domain $R$ satisfying (ID$_2$) is necessarily a B\'ezout domain''. Unique factorization domains, projective-free domains and PRINC domains (cf.\cite{SalZan}) verify the conjecture. We prove that an integral domain $R$ satisfying (ID$_2$) must be a Pr\"ufer domain in which every invertible $2\times 2$ matrix is a product of elementary matrices. Then we show that a large class of coordinate rings of plane curves and the ring of integer-valued polynomials Int($\mathbb{Z}$) verify an equivalent formulation of (C).

\end{abstract}

\maketitle

\section{Introduction.}

In this paper we deal with a classical problem on the factorization of square matrices over rings: characterize the integral domains $R$ that, for any integer $n>0$, satisfy the following property
\begin{enumerate}
\item[\IDn :] Every singular $n\times n$ matrix with entries in $R$ is a product of idempotent matrices.
\end{enumerate}

In a commutative setting, this natural problem was firstly attacked in the 1967 paper by J.A. Erdos \cite{Erdos}, who proved that every singular $n \times n$ matrix over a field can be written as a product of idempotent matrices. In 1983 Laffey \cite{Laff1} proved that also Euclidean domains satisfy property {\IDn}, for every $n>0$.  In 1991 Fountain \cite{Fount} studied the property {\IDn} in the class of principal ideal domains. He found some properties equivalent to {\IDn} and re-derived that $\mathbb{Z}$ and discrete valuation rings satisfy {\IDn} for all $n>0$. Many other papers were devoted to analogous problems for non-commutative rings. For instance, Hannah and O'Meara \cite{HannahOmeara} characterized products of idempotents in some classes of von Neumann regular rings. In recent years much research on this topic was made in the non-commutative setting, by several authors: Alahmadi, Facchini, Jain, Lam, Leroy, Sathaye (see \cite{AJL_1}, \cite{AJLL}, \cite{AJLS_Electronic}, \cite{AJL_2}, \cite{FL}).

We mention another classical problem on matrices factorization, somehow symmetric to the previous one: characterize the integral domains $R$ that satisfy the following property
\begin{enumerate}
\item[\GEn :] Every invertible $n\times n$ matrix with entries in $R$ is a product of elementary matrices.
\end{enumerate}

Property {\GEn} was introduced and investigated in Cohn's seminal paper \cite{Cohn}. 
When $R$ is a B\'ezout domain, a strong connection between the two properties is given by a fundamental theorem proved in 1993 by Ruitenburg \cite{Ruit}. Generalizing the results by Fountain, Ruitenburg showed that a Hermite (not necessarily commutative) domain satisfies property {\IDn} for every $n>0$ if and only if it satisfies property {\GEn} for every $n>0$. Other relations between {\IDn} and {\GEn} for non-commutative rings may be found in \cite{AJLL} and \cite{FL}. 

In the very recent paper \cite{CZZ}, the authors and U. Zannier have investigated property {\GEn} in connection with non-Euclidean principal ideal domains. Actually, some arguments in Section 4.1 are adapted from \cite{CZZ}.

In this paper we focus on {\ID2}, since this case is crucial for the general case {\IDn}. Indeed, in the class of B\'ezout domains two {\it lifting properties} from $2 \times 2$ to $n \times n$ matrices hold, for any $n > 0$. The first one, proved by Laffey (cf. \cite{Laff1}) for Euclidean domains and extended, almost {\it verbatim}, to PID's (cf. \cite{BhasRao}) and to B\'ezout domains (cf. \cite{SalZan}), states that a B\'ezout domain satisfies the property {\ID2} if and only if it satisfies {\IDn} for every $n > 0$. The second one is Theorem 7.1 of Kaplansky's paper \cite{Kap_el_div}: a B\'ezout domain satisfies the property {\GE2} if and only if it satisfies {\GEn} for all $n>0$. In particular, by Ruitenburg's theorem, {\ID2} and {\GE2} are equivalent over a B\'ezout domain.

However, it is crucial to remark that this equivalence is no longer valid outside the class of B\'ezout domains: for instance, Cohn proved (Theorem 4.1 of \cite{Cohn}) that any local domain satisfies property {\GE2}, while a local domain satisfying {\ID2} is necessarily a valuation domain (see \cite{SalZan} for the easy proof; see also Remark \ref{not-equivalent}). 

Bhaskara Rao \cite{BhasRao} recently showed that if any singular $2 \times 2$ matrix over a projective-free integral domain $R$ is a product of idempotent matrices, then $R$ is a B\'ezout domain. The results in \cite{BhasRao}, \cite{Laff1}, \cite{Ruit} suggested the following conjecture, proposed in \cite{SalZan}:
\begin{enumerate}
\item[(C)] ``If an integral domain $R$ satisfies property \ID2, then it is a B\'ezout domain.''
\end{enumerate}

Important classes of domains have been shown to satisfy conjecture (C), see \cite[Sect.\,4]{SalZan}: unique factorization domains, projective-free domains and PRINC domains. This latter class of domains, introduced by Salce and Zanardo and investigated in \cite{PerSalZan}, properly contains the other two classes. The notion of PRINC domain is strongly related to that of {\it unique comaximal factorization domain}, studied by McAdam and Swan in \cite{UCFD}.

We remark that, if the conjecture (C) is true, then, by the lifting property of B\'ezout domains, every integral domain $R$ satisfying {\ID2} would also satisfy {\IDn} for all $n>0$, and, by Ruitenburg's theorem, $R$ would also satisfy {\GEn} for every $n>0$. This further justifies our restriction to {\ID2}.

The purpose of this paper is to investigate the above conjecture, showing that we may confine ourselves to Pr\"ufer domains and finding important classes of domains satisfying (C). 

In Theorem \ref{ID2_implica_Prufer}, we prove that if an integral domain $R$ satisfies property {\ID2}, then $R$ is a Pr\"ufer domain. Therefore we can give an equivalent formulation of (C), namely,

\begin{enumerate}
\item[(C')] ``If $R$ is a Pr\"ufer domain that is not a B\'ezout domain, then it does not satisfy property {\ID2}".
\end{enumerate}

We find a general relation between properties {\ID2} and {\GE2} over any integral domain $R$. Namely, in Proposition \ref{ID2implicaGE2} we prove that if every singular $2\times 2 $ matrix over an integral domain $R$ is a product of idempotent matrices, then every invertible $2\times 2 $ matrix over $R$ is a product of elementary matrices. 

It follows that every Pr\"ufer non-B\'ezout domain, not satisfying {\GE2}, verifies the conjecture (C').

In this setting, we find a large class of algebraic curves whose coordinate rings satisfy (C'). We prove that if $\c_0$ is an affine non-singular plane curve over a field $k$, such that its points at infinity are not rational over $k$ and all conjugate by elements of the Galois group $G_{\bar{k}/k}$, then the coordinate ring $R$ of $\c_0$ does not satisfy property {\GE2}. Hence, when $R$ is not a PID, it is a Dedekind domain (so also a Pr\"ufer domain) that verifies the conjecture (C').

Moreover, properly applying some results by Cohn on discretely ordered rings (cf. \cite[Section 8]{Cohn}), in Theorem \ref{main4} we prove that the ring $\text{Int}(\Z)$ of integer-valued polynomials, one of the most important examples of Pr\"ufer domain that is not B\'ezout, verifies (C').

\section{Pr\"ufer domains and property \ID2.}

In what follows $R$ will always be a commutative integral domain. For any given ring $A$, we will denote by $A^*$ its multiplicative group of units. 

We call an integral domain $R$ {\it local} if it contains a unique maximal ideal, we do not require $R$ to be Noetherian. 
As well-known, an integral domain $R$ is said to be: a B\'ezout domain if every finitely generated ideal of $R$ is principal; a Pr\"ufer domain if every finitely generated ideal of $R$ is invertible.

We recall an easy characterization of the idempotent matrices of dimension $2$ (see \cite{SalZan}). Every $2\times 2$ non-zero non-identity idempotent matrix over an integral domain $R$ is of the form $\begin{pmatrix}
x & y\\
z & 1-x
\end{pmatrix}$, with $x(1-x)= yz$. 

We say that $x, y \in R$ are an {\it idempotent pair} if $(x \ y)$ is the first row of an idempotent matrix. It is important to observe that the $R$-ideal $(x, y)$ generated by an idempotent pair is invertible (cf. \cite{PerSalZan}, Theorem 1.3). In fact, the following equalities hold:
$$
(x, y) (1 - x, y) = (yz, xy, y(1-x), y^2) = y(z, x, 1-x, y) = yR.
$$

\begin{Remark} \label{not-equivalent}
A conjecture analogous to (C), but replacing {\ID2} with {\GE2}, is false: there exist non-B\'ezout domains satisfying the property {\GE2}. For example, take $R$ a local non-valuation domain. Then $R$ is not a B\'ezout domain, and does not satisfy {\ID2} by Corollary 5.3 of \cite{SalZan}. However, $R$ satisfies the property {\GEn}, for any $n > 0$. In fact, if $M$ is an invertible $n \times n$ matrix with entries in $R$, then every column of $M$ has an entry which is a unit, since $R$ is local and the determinant is a unit of $R$. Then, by Gauss elimination, $M$ is equivalent to an invertible diagonal matrix through multiplication by elementary matrices, hence it is a product of elementary matrices. 
\end{Remark}

In this section we prove that an integral domain that satisfies {\ID2} must be a Pr\"ufer domain. Therefore, when investigating conjecture (C), we may confine ourselves to the class of Pr\"ufer domains.

\begin{Proposition}\label{lemmaprufer}
Let $R$ be an integral domain and $a,b$ two non-zero elements of $R$. If the matrix $\begin{pmatrix}
a & b\\
0 & 0 
\end{pmatrix}$ is a product of idempotent matrices, then the ideal $(a,b)$ is an invertible ideal of $R$.
\end{Proposition}

\begin{proof}
Assume that 
\begin{equation*}
\begin{pmatrix}
a & b\\
0 & 0 
\end{pmatrix}=\begin{pmatrix}
p & q\\
r & s
\end{pmatrix}\begin{pmatrix}
x & y\\
z & 1-x
\end{pmatrix}
\end{equation*}
where $\begin{pmatrix}
p & q\\
r & s
\end{pmatrix}$ is a product of idempotent matrices and $\begin{pmatrix}
x & y\\
z & 1-x
\end{pmatrix}$ is a non-identity idempotent matrix.

The above equality yields the linear system (in the unknowns $p, q$):
\begin{equation}\label{a}
\begin{cases}
x p + z q = a\\
y p + (1 - x) q= b
\end{cases}
\end{equation}
Since the system (\ref{a}) is solvable and its matrix is singular,  we get   

\begin{equation}\label{rel}
 a/b = x/y = z/(1-x).
\end{equation}
In particular, the entries $x, y, z, 1-x$ are all nonzero.

It follows from (\ref{rel}) that
$$
y ( a, b) =( ay, by ) = ( bx, by )= b ( x, y).
$$
Since $x,y$ is an idempotent pair, the ideal $( x,y )$ is invertible, as observed above. We conclude that $( a, b)$ is also invertible.
\end{proof}

A well known result by Gilmer (cf. \cite[Th. 22.1]{Gilmer}) says that $R$ is a Pr\"{u}fer domain if and only if every two-generated ideal of $R$ is invertible, hence from the preceding proposition we readily derive the following useful result.

\begin{Theorem}\label{ID2_implica_Prufer}
If $R$ is an integral domain satisfying property \ID2, then $R$ is a Pr\"ufer domain.
\end{Theorem}

\begin{proof}
Assume that $R$ satisfies property \ID2. Then, every matrix of the form $\begin{pmatrix}
a & b\\
0 & 0
\end{pmatrix}$ with $a,b$ non-zero elements of $R$, is a product of idempotent matrices. Thus, it follows from Proposition \ref{lemmaprufer} that every two-generated ideal of $R$ must be invertible, i.e. $R$ must be a Pr\"ufer domain.
\end{proof}

The preceding theorem provides an equivalent formulation of the conjecture (C):

\begin{enumerate}
\item[(C')] Every Pr\"ufer non-B\'ezout domain $R$ does not satisfy property {\ID2}.
\end{enumerate}

\section{Property {\ID2} implies {\GE2}.}

As recalled in the introduction, we know that the factorization properties {\IDn} and {\GEn} are equivalent inside the class of B\'ezout domains. 
In this section we prove that property {\ID2} implies property {\GE2} over an arbitrary integral domain.

\medskip

Let us recall an important result due to Kaplansky \cite{Kap}:

\begin{Lemma}[Lemma 1 of \cite{Kap}]\label{Kaplemm} 
Let $R$ be an integral domain and let $I_1,\dots,I_m,$ $J_1,\dots ,J_m$ be (integral or fractional) ideals of $R$ such that
\[ I_1\oplus \cdots \oplus I_m \cong J_1\oplus \cdots \oplus J_m\]
as $R$-modules. Then
\[I_1\cdots I_m \cong J_1 \cdots J_m.\]
\end{Lemma}

From the above lemma, we immediately get the following corollary

\begin{Corollary}\label{corKap} 
Let $R$ be an integral domain and $J$ a fractional ideal of $R$. If $R\oplus R \cong J\oplus R$, then $J$ is a free $R$-module
\end{Corollary}

\begin{proof} 
Lemma \ref{Kaplemm} shows that $R\cdot R \cong J\cdot R$, hence $J\cong R$.
\end{proof}

We are now in the position to prove the next result.

\begin{Proposition}\label{proposition_H2} 
Let $A$ and $B$ be free direct summands of rank one of the free $R$-module $R^2$ with $A \cap B=0$. Then there exists an endomorphism $\beta$ of $R^2$ with $Ker(\beta)=B$ and $Im(\beta)=A$.
\end{Proposition}

\begin{proof} 
Let $A$ and $B$ be free direct summands of rank one of $R^2$, i.e. $R\oplus R= A\oplus A'= B\oplus B'$ with $A\cong R$ and $B\cong R$. Hence, by Corollary \ref{corKap} $A'$ and $B'$ are also free of rank one. Let $p: B \oplus B' \to B'$ be the canonical projection with kernel $B$, and $f : B' \to A$ an isomorphism. Let $g$ be the endomorphism of $R^2$ defined by $g: b + b' \mapsto f(b')$, for $b \in B$, $b' \in B'$. Then the endomorphism $\b = g \circ p$ has kernel $B$ and image $A$. 
\end{proof}

Let us recall that, by Theorem 3.4 of \cite{SalZan}, an integral domain $R$ satisfies {\GEn}, with $n>0$, if and only if it satisfies the following property:
\begin{enumerate}
\item[(HF$_n$)] For any free direct summand $A,B$ of the free $R$-module $R^n$, of ranks $r$ and $n-r$ respectively ($1\leq r \leq n$), with $A \cap B = 0$, there exists an endomorphism $\beta$ of $R^n$ with Ker$(\beta)=B$ and Im$(\beta)=A$, that is a product of idempotent endomorphisms of rank $r$.
\end{enumerate}

We get the following important consequence of Proposition \ref{proposition_H2} and Theorem 3.4 in \cite{SalZan}.

\begin{Proposition}\label{ID2implicaGE2}
If $R$ satisfies ($\text{ID}_2$), then it also satisfies ($\text{GE}_2$).
\end{Proposition}

\begin{proof} 
Let $A$, $B$ be direct summands of $R^2$ of rank $1$, with $A \cap B = 0$. By Proposition \ref{proposition_H2} there exists an endomorphism $\beta$ of $R^2$ with $Ker(\beta)=B$ and $Im(\beta)=A$. By property {\ID2} the singular matrix associated to $\beta$ (with respect to some basis) is a product of idempotent matrices, hence $\beta$ is a product of idempotent endomorphisms. Therefore property ($\text{HF}_2$) holds, hence Theorem 3.4 of \cite{SalZan} shows that ($\text{GE}_2$) holds, as well.
\end{proof}

From Proposition \ref{ID2implicaGE2} and Theorem \ref{ID2_implica_Prufer}, we immediately get
 
\begin{Corollary}\label{final_corollary}
If $R$ is an integral domain satisfying property {\ID2}, then $R$ is a Pr\"ufer domain satisfying property {\GE2}.
\end{Corollary}

The preceding corollary shows that every Pr\"ufer non-B\'ezout domain, not satisfying {\GE2}, does not satisfy {\ID2}, hence verifies the conjecture (C').

\section{Classes of rings verifying (C').}

In 1966, Cohn \cite[Theorem 6.1]{Cohn} proved that the rings of integers $O$ in $\mathbb{Q}(\sqrt{-d})$, with $d$ a squarefree positive integer, do not satisfy property {\GE2}, unless $d=1,2,3,7,11$ (values for which $O$ is an Euclidean domain). If $d$ is also different from $19,43,67, 163$, then the rings of integers $O$ in $\mathbb{Q}(\sqrt{-d})$ are not principal ideal domains. Thus they are Dedekind domains, not UFD and so non-B\'ezout, that do not satisfy the property {\GE2}, thus verifying (C'). 

\medskip

In this section we prove that also the coordinate rings of a large class of plane curves and the ring of integer-valued polynomials $\text{Int}(\Z)$ verify the conjecture (C').

\subsection{The case of the coordinate rings.}

We fix the notation and recall some basic facts on coordinate rings. We refer to \cite{Fulton} and \cite{Silverman} for more details and proofs. 

With $k$ we will denote a perfect field, with $\bar{k}$ its algebraic closure, with $\c$ a smooth projective curve over $k$, with $\c_0$ an affine part of $\c$ and with $\c_\infty$ the corresponding set of points at infinity. The affine coordinate ring of $\c_0$ and its quotient field, namely, the function field of $\c_0$, are denoted respectively as $R=k[\c_0]$ and $k(\c_0)$. It is well known that $k[\c_0]$ is a Dedekind domain. 

A $k$-rational divisor $D$ of the smooth curve $\c$ is the sum
\[D=\sum_{P\in\c}n_P P\]
with $n_P\in\Z$, almost all $n_P=0$, and $n_P=n_Q$ if $P=gQ$ for some $g\in G_{\bar{k}/k}$. We will call $\text{Div}_k(\c)$ the set of all $k$-rational divisor of $\c$.
A divisor $D\in\text{Div}_k(\c)$ is said to be \textit{principal} if there exists $F\in k(\c_0)^*$ such that
\[D= {\rm div}(F)=\sum_{P\in\c}\text{ord}_P(F) P.\]

Recall the following
\begin{Proposition}\label{gradodivisore} Let {\C} be a smooth curve and $F\in k(\c_0)^*$. Then:
\begin{enumerate}[(i)]
\item $\div(F)=0$ iff $F\in k^*$.
\item $\deg(\div(F))=0$.
\end{enumerate}
\end{Proposition}

Let us consider a smooth projective curve $\c$ over the perfect field $k$ and its affine coordinate ring $R=k[\c_0]$. Define, for any $z\in R$
\begin{equation}\label{d(z)}
d(z)= -\sum_{P\in C_{\infty}}\ord_P(z).
\end{equation}

We will need some lemmas.

\begin{Lemma}\label{kring}
Let $R = k[\c_0]$ be the affine coordinate ring of the smooth curve $\c$ over the field $k$. If the points at infinity of ${\c}$ are non-rational and all conjugate by elements of the Galois group $G_{\bar{k}/k}$, then $R^*=k^*$.
\end{Lemma}

\begin{proof}
Let us assume that all the elements of $\c_\infty$ are non-rational and conjugate by elements of $G_{\bar{k}/k}$. It follows that any nonzero rational divisor at infinity has the form $m \sum_{P \in \c_\infty} P$, for some nonzero integer $m$. Let $u$ be a unit of $R$. Then $u$ has no zeroes in $\c_0$, hence $\div(u)$ is a divisor at infinity. Then $\deg(\div(u)) = 0$ implies $\div(u) = 0$, hence $u \in k^*$ by Proposition \ref{gradodivisore} (i). 
\end{proof}

\begin{Lemma}\label{dfunction}
In the above notation, if the points at infinity of ${\c}$ are non-rational and all conjugate by elements of the Galois group $G_{\bar{k}/k}$, then the map $d : R \to \mathbb{N} \cup \{ - \infty \}$ defined by {\rm (\ref{d(z)})} satisfies the following properties:
\begin{enumerate}
\item[(d1)] $d(z)= -\infty$ if and only if $z=0$,
\item[(d2)] $d(z) = 0$ if and only if $z \in k^*$,
\item[(d3)] $d(z + t) \leq \max \{d(z), d(t)\}$,
\item[(d3')]$d(z + t) = \max \{d(z), d(t)\}$ whenever $d(z) \neq d(t)$
\item[(d4)] $d(zt)= d(z)+d(t)$,
\end{enumerate}
for any $z,t\in R$
\end{Lemma}
\begin{proof}
Since the points at infinity are non-rational and conjugate to each other, the valuation $v = \ord_P$ on $k(\c)$ does not depend on the choice of $P \in \c_\infty$: if $\c_{\infty}=\{P_1,\dots,P_m\}$, then $\ord_{P_1}(F)=\dots=\ord_{P_m}(F)$ for any $F\in k(\c)$. Then the map $d$ on $R$ coincides with $- m v$, where $m = | \c_\infty|$. It is straightforward to show that $d = - m v$ satisfies properties (d1) and (d4). Indeed, the properties (d3) and (d3') hold since $d$ is the opposite of a valuation, (d2) holds since $R^*=k^*$ by Lemma \ref{kring}. To conclude we remark that the map $d$ actually takes values in $\mathbb{N} \cup \{ - \infty \}$. Any element $z\in k[\c_0] \setminus \{0\} $ has no poles outside $\c_\infty$, therefore ${\rm ord}_P(z)\geq 0$ for all $P\in \c_0$. Then, from Proposition \ref{gradodivisore}, it follows that $d(z)= -\sum_{P\in \c_{\infty}}\ord_P(z)=\sum_{P\in \c_0}\ord_P(z)\geq 0$ for any non-zero $z\in k[\c_0]$.
\end{proof}

\begin{Remark}
It is worth noting that Lemma \ref{kring} is nothing but one of the two implications in \cite[Lemma 3.4]{CZZ}; here we do not need to assume that $R$ is a PID. Moreover Lemma \ref{dfunction} can be obtained as a consequence of the above Lemma \ref{kring} and of Lemma 3.3 in \cite{CZZ}. We have given a direct proof for the sake of completeness.
\end{Remark}

\medskip

In accordance with the terminology in \cite{CZZ} and following Cohn \cite{Cohn} we say that a ring $R$ containing the field $k$ as a subring is a $k$-ring if $R^*= k$. Moreover if $R$ is a ring containing $k$, a map $\delta : R \to \mathbb{N} \cup \{ - \infty \}$ satisfying properties (d1)--(d4) is called a {\it pseudo-valuation}. We observe that (d2) and (d4) imply that $R^* = k^*$, hence $R$ is a $k$-ring; moreover (d1) and (d4) show that $R$ cannot contain non-trivial zero-divisors, hence it is an integral domain.

With this notation, from Lemmas \ref{kring} and \ref{dfunction}, we get that the affine coordinate ring of a smooth curve $\c$ over the perfect field $k$ having all the points at infinity non-rational over $k$ and conjugate by elements of the Galois group $G_{\bar{k}/k}$ is a $k$-ring with pseudo-valuation $d$, defined by (\ref{d(z)}).

\medskip

Two elements $a,b$ of a $k$-ring $R$ with pseudo-valuation $\delta$ are said to be \textit{$R$-independent} if for any non-zero $c\in R$ we have
\[\delta(a + bc) \ge \delta(a), \quad \delta(b + ac) \ge \delta(b).\]

We say that a pair of elements $a, b \in R$ form a {\it regular row} if $(a \ b)$ is the first row of a suitable $2 \times 2$ invertible matrix with entries in $R$.

Let us recall the following proposition due to P.M. Cohn.

\begin{Proposition}[Proposition 7.3 of \cite{Cohn}]\label{Cohn} 
Let $R$ be a $k$-ring with a pseudo-valuation $\delta$ that satisfies the property (GE$_2$), and let $a$, $b$ be elements of $R$ such that $\delta(a) = \delta(b)$. If $a, b$ form a regular row, then they cannot be $R$-independent.
\end{Proposition} 

We are now able to prove the next theorem.

\begin{Theorem}\label{main3}
Let $\c\subset\P^2$ be a plane smooth curve over the field $k$, having degree $\geq 2$, such that the points at infinity are non-rational and all conjugate to each other by elements of the Galois group $G_{\bar{k}/k}$. Then $R=k[\c_0]=k[\c\setminus \c_\infty]$ does not satisfy property (GE$_2$). 
\end{Theorem}

\begin{proof} 
Since all the elements of $\c_\infty$ are non-rational and conjugate to each other, then from Lemma \ref{kring}, $R$ is a $k$-ring. 

The remainder of the argument is analogous to the proof of Theorem 3.5 in \cite{CZZ}.

Let $F \in k[X, Y]$ be a polynomial of degree $n \ge 2$ and $F(X, Y)= 0$ be the defining equation of $\c_0$. We can assume, without loss of generality, that $F(0, 0) \neq 0$. Let $F_n(X, Y)$ be  the homogeneous component of $F$ of degree $n$. Since the points at infinity are conjugate and not rational, it follows that $F_n(X, Y) = c \prod_{i = 1}^n(Y - \a_i X)$, where $c \in k$, $\a_i \in \bar{k} \setminus k$, and $P_i = (1, \a_i, 0)$ are the points at infinity,  $1 \le i \le n$. Since the $P_i$'s are conjugate by elements of $G_{\bar{k}/k}$, from Lemma \ref{dfunction} we get that the map $d = - \sum_{i = 1}^n \ord_{P_i}$ satisfies properties (d1)-(d4). 

We now consider the elements $x, y$ of the coordinate ring $R$. Since $F(0, 0) \neq 0$, it is easily seen that they can occur as a first row of an invertible matrix in $M_2(R)$ so they form a regular row. Considering homogeneous coordinates, it is also straightforward to verify that $d(x) = d(y)$. It remains to verify that $x, y$ are $R$-independent. Take any nonzero $z \in R$. If $z \notin k^*$, than $d(z) > 0$ and so $d(x + yz) > d(x), d(y + xz) > d(y)$ by the properties of $d$. If $z \in k^*$, it is easy to see that $d(x + yz) = d(x) = d(y) = d(y + xz)$. Now we can apply Proposition \ref{Cohn}, and conclude that $R$ does not satisfy property (GE$_2$).   
\end{proof}

Therefore, given a plane smooth curve $\c$ of degree $\geq 2$ having non-rational conjugate points at infinity, whenever its coordinate ring $R$ is not a principal ideal domain, then $R$ is a Dedekind domain (so a Noetherian Pr\"ufer domain) that satisfies (C').
 
\begin{example}
Let us consider the coordinate ring $R=\R[\c_0]$ of the affine smooth curve $\c_0$ over $\R$ having defining equation $x^4+y^4+1=0$. Then $R$ is a non-UFD Dedekind domain. This can be seen by observing that 
\[(x^2+y^2-1)(x^2+y^2+1)=2(xy-1)(xy+1)\]
is a non-unique factorization into indecomposable factors. Moreover $R$ does not satisfy property {\GE2} by Theorem \ref{main3}. 
\end{example}

\subsection{$\text{Int}(\Z)$ and the property \ID2.} 

A natural example of Pr\"ufer domain that fails to be a B\'ezout domain is the celebrated ring of integer-valued polynomials $\text{Int}(\Z)$. In this section we prove that this ring does not satisfy property {\GE2} thus verifying the conjecture (C) in its equivalent formulation (C').

The ring $\text{Int}(\Z)$ of integer-valued polynomials is defined as the set of rational polynomials taking integral values on integers:
$$
{\rm Int}(\Z)=\{f\in\Q[X]\,|\,f(\Z)\subseteq \Z\}.
$$

It is clear that $\text{Int}(\Z)$ is a $\Z$-module such that $\Z[X]\subseteq \text{Int}(\Z)\subseteq \Q[X]$.

Since the 1970s, integer-valued polynomials and their various generalizations have been deeply studied, revealing many interesting features. A main reference for this topic is the volume by Cahen and Chabert \cite{Integer-valued}; see also \cite{Chabert}. For our purposes it is enough to recall the following two results.

\begin{Proposition} 
The $\Z$-module ${\rm Int}(\Z)$ is free.
The polynomials
\[{X\choose n}=\frac{X(X-1)\cdots(X-n+1)}{n!},\]
with the convention ${X\choose 0}=1$ and ${X\choose 1}=X$, form a basis of ${\rm Int}(\Z)$.
\end{Proposition}

\begin{proof}
See \cite[Prop. I.1.1.]{Integer-valued}.
\end{proof}

\begin{Proposition}
The ring ${\rm Int}(\Z)$ of integer-valued polynomials is a Pr\"ufer domain. Moreover, it is not a B\'ezout domain.
\end{Proposition}

\begin{proof}
The result is an immediate consequence of \cite[Prop.6.3]{Chabert}, or equivalently of \cite[Th.VI.1.7]{Integer-valued}.
\end{proof}

In accordance with Cohn \cite[Sect.8]{Cohn}, we say that a ring $R$ is \textit{discretely ordered} if it is totally ordered and, for any $r\in R$, if $r>0$, then $r\geq 1$. The ring of integers $\Z$ is the most obvious example of a discretely ordered ring.

\begin{Proposition}
The ring of integer-valued polynomials ${\rm Int}(\Z)$ is a discretely ordered ring.
\end{Proposition}

\begin{proof}
Let $f=a_0+a_1X+\dots+a_n{X\choose n}$ ($a_i \in \Z$) be any element of $\text{Int}(\Z)$. We will say that $f>0$ if and only if $a_n>0$. Then, given any $f,g\in\text{Int}(\Z)$, we have $f>g$ if and only if $f-g>0$. Moreover, if $f>0$, then it is clear from the definition of the order relation that it must be $f\geq 1$.
\end{proof}

We now summarize some results from \cite[Sect.8]{Cohn} that relate discretely ordered rings and the property {\ID2}.
As usual, the group of invertible $2 \times 2$ matrices over $R$ is denoted by $GL_2(R)$.

\begin{Theorem}[cf. Theorem 8.2 of \cite{Cohn}]\label{uniqueform}
Let $R$ be a discretely ordered ring. Then any $\mathbf{M} \in GL_2(R)$ that is a product of elementary matrices can be uniquely written in the form 
\begin{equation}
\mathbf{M}=\begin{pmatrix}
\alpha & 0\\
0 & \beta
\end{pmatrix}\mathbf{T}(r_1)\cdots\mathbf{T}(r_k),
\end{equation}
where $\alpha,\beta\in R^*$,
\[\mathbf{T}(r_i)=\begin{pmatrix}
r_i & 1\\
1 & 0
\end{pmatrix},\]
and the $r_i\in R$ satisfy
\[r_1\geq 0,\quad r_i>0\quad\text{for }1<i<k,\]
and when $k=2$, $r_1,r_2$ are not both zero.
\end{Theorem}

Let $t_1,t_2,\dots $ be a sequence of (non-commuting) indeterminates and define recursively a sequence of polynomials $p_k$ in the $t_i$'s, as follows: $p_{-1}=0$, $p_0=1$, and
$$
 p_k(t_1, \dots,t_k)=p_{k-1}(t_1,\dots,t_{k-1})t_k+p_{k-2}(t_1,\dots,t_{k-2}).
$$

Then 
\begin{equation} \label{product of T}
\mathbf{T}(r_1)\cdots \mathbf{T}(r_k)=\begin{pmatrix}
p_k(r_1,\dots,r_k) & p_{k-1}(r_1,\dots,r_{k-1})\\
p_{k-1}(r_2,\dots,r_k)& p_{k-2}(r_2,\dots,r_{k-1})
\end{pmatrix}.
\end{equation}
This can be easily seen by induction.

\begin{Lemma}[cf. Lemma 8.3 in \cite{Cohn}]\label{lemmap}
Let $R$ be a discretely ordered ring and $r_1,\dots,r_k\in R$. If $r_1\geq 0$, $r_i>0$, $1< i < k$ with $k\geq 2$, then
\[p_k(r_1,\dots,r_k)> p_{k-1}(r_1,\dots,r_{k-1}).\]
\end{Lemma}

Let $\mathbf{M}$ be as in Theorem \ref{uniqueform}. Then by (\ref{product of T}) we get
\begin{equation}\label{M}
\mathbf{M}=\begin{pmatrix}
a & b\\
c & d
\end{pmatrix}=\begin{pmatrix}
\alpha\, p_k(r_1,\dots, r_k) & \alpha\, p_{k-1}(r_1,\dots,r_{k-1})\\
\beta \, p_{k-1}(r_2,\dots, r_k) & \beta\, p_{k-2}(r_2,\dots,r_{k-1})
\end{pmatrix},
\end{equation}
with $\alpha,\beta \in R^*$ and $r_1,\dots,r_k\in R$ such that $r_1\geq 0$ and $r_i>0$, $1<i<k$.
Moreover, from (\ref{M}) and the definition of $p_k$, we also get
\begin{equation}
\begin{split}
\mathbf{M}(\mathbf{T}(r_k))^{-1}&=\begin{pmatrix}
a & b\\
c & d
\end{pmatrix}\begin{pmatrix}
0 & 1\\
1 & -r_{k}
\end{pmatrix}
=\begin{pmatrix}
b & a-br_k \\
d & c-dr_k
\end{pmatrix}\\
&=\begin{pmatrix}
\alpha\,p_{k-1}(r_1,\dots, r_{k-1}) & \alpha\,p_{k-2}(r_1,\dots, r_{k-2})\\
\beta\,p_{k-2}(r_2,\dots, r_{k-1}) & \beta\,p_{k-3}(r_2,\dots, r_{k-2})
\end{pmatrix}
\end{split}
\end{equation}
Therefore, when $k\geq 2$, Lemma \ref{lemmap} implies that $b=\alpha\,p_{k-1}(r_1,\dots, r_{k-1})$ and $a-br_k=\alpha\,p_{k-2}(r_1,\dots, r_{k-2})$ must have the same sign, depending on $\alpha$. Say $b, a-br_k>0$ (otherwise it suffices to replace $\mathbf{M}$ with $-\mathbf{M}$). Thus, from Lemma \ref{lemmap}, 
\begin{equation}\label{sign of b}
b>a-br_k>0.
\end{equation}

\begin{Lemma}[cf. Lemma 8.4 of \cite{Cohn}]\label{Lemma finale} 
Let $R$ be a discretely ordered ring. Take $\mathbf{M} = \begin{pmatrix}
a & b\\
c & d
\end{pmatrix}  \in GL_2(R)$ which is a product elementary matrices. Let us write
\[\mathbf{M}=\begin{pmatrix}
\alpha & 0\\
0 & \beta
\end{pmatrix}\mathbf{T}(r_1)\cdots\mathbf{T}(r_k)\]

as in Theorem \ref{uniqueform}, and assume that $k\geq 2$ and $b > 0$. Then
\begin{enumerate}[(i)]
\item if $r_k>0$, then $a>b>0$;
\item if $r_k=0$, then $b>a>0$;
\item if $r_k=-c<0$, then $b>a>-bc$.
\end{enumerate}
\end{Lemma}

Let us remark that all the above results are true for the discretely ordered ring $\text{Int}(\Z)$. 

\medskip

We are finally able to prove the following

\begin{Theorem}\label{main4}
The ring of integer-valued polynomials ${\rm Int}(\Z)$ does not satisfy property {\rm {\GE2}}.
\end{Theorem}

\begin{proof}
 Let us consider the following matrix of $GL_2(\text{Int}(\Z))$:
\[\mathbf{M}=\begin{pmatrix}
1+2X & 4\\
1+4X+2{X \choose 2} & 5+2X
\end{pmatrix}.\]

Let us assume, for a contradiction, that $\mathbf{M}$ is a product of elementary matrices. Then, by Theorem \ref{uniqueform}, $\mathbf{M}=\begin{pmatrix}
\alpha & 0\\
0 & \beta
\end{pmatrix}\mathbf{T}(r_1)\cdots\mathbf{T}(r_k)$ with $\alpha,\beta\in R^*$, $r_i\in R$, such that
$r_1\geq 0$, $r_i>0$ for $1<i<k$. Note that $k\geq 2$; in fact for $k=0$, $\mathbf{M}=\begin{pmatrix}
\alpha & 0\\
0 & \beta
\end{pmatrix}$ while, for $k=1$, $\mathbf{M}=\begin{pmatrix}
\alpha r_1 & \alpha\\
\beta & 0
\end{pmatrix}$. Moreover, if we set $(a,b)=(1+2X , 4)$, then $b>0$ and $a>b>0$. Therefore, we are in case (i) of Lemma \ref{Lemma finale}, and it must be $r_k>0$. But from (\ref{sign of b}) we also have $b>a-br_k>0$, in particular 
\[a>br_k.\]
Let $r_k=a_{0k}+a_{1k}X+\dots+a_{mk}{X\choose m}$. Thus, since $r_k>0$, then $a_{mk}>0$, hence $a_{mk}\geq 1$ and $4a_{mk}>2$. This shows that $r_k$ must be an element of $\Z$ otherwise we would get $a-br_k<0$. But for such $r_k$ we have $a-br_k=(1-4a_{0k})+2X>4=b$, thus contradicting (\ref{sign of b}). We conclude that $\text{Int}(\Z)$ does not satisfy property {\GE2}.
\end{proof}

\section*{Acknowledgements}

The authors would like to thank the referee for valuable comments and suggestions, that helped to considerably improve the first version of this paper.

\bibliographystyle{plain}
\def\cprime{$'$} \def\cprime{$'$}

\end{document}